\documentclass{article}
\usepackage{amsmath,amsfonts,amsthm,mathrsfs,indentfirst}
\usepackage[left=3cm,right=3cm,top=2.5cm,bottom=2.5cm]{geometry}
\usepackage{graphicx,subfigure}
\usepackage{epstopdf}
\usepackage{authblk}
\usepackage[numbers,sort&compress]{natbib}
\usepackage{hyperref,mathrsfs}
\usepackage{colortbl}

\def\cor#1{{\color{red}#1}}

\newtheorem{thm}{Theorem}
\newtheorem{prop}{Proposition}
\newtheorem{lemma}{Lemma}

\newtheorem{defn}{Definition}
\newtheorem{remark}{Remark}

\def\d{\mathrm{d}}
\def\bP{\mathbb{P}}
\def\hbP{\widehat{\mathbb{P}}}
\def\Tau{\mathscr{T}}
\def\V{\mathscr{V}}
\def\E{\mathscr{E}}
\def\G{\mathscr{G}}
\def\pa{\partial}
\def\R{\mathbb{R}}
\def\hR{\widehat{R}}
\def\T{\mathscr{T}}
\def\V{\mathscr{V}}

\def\M{\mathscr{M}}

\def\B{\mathscr{B}}
\def\O{\Omega}
\def\gss{\sqrt{3/5}}
\def\balpha{\boldsymbol\alpha}
\def\bnu{\boldsymbol\nu}
\def\bxi{\boldsymbol\xi}
\def\hphi{{\widehat\phi}}

\def\NC{\mathscr{NC}^h}
\def\Span{\mathrm{Span}}
\renewcommand\hat{\widehat}

\allowdisplaybreaks
\begin{document}
\title{A new cubic nonconforming finite element on rectangles\thanks{This project is supported by
NNSFC (Nos. 61033012,19201004,11271060,61272371), ``the Fundamental
Research Funds for the Central Universities'', and NRF Nos. 2012-0000153.}}
\author[a]{Zhaoliang Meng\thanks{Corresponding author: mzhl@
dlut.edu.cn}}
\author[a,b]{Zhongxuan Luo}
\author[c]{Dongwoo Sheen}
\affil[a]{\it School of Mathematical Sciences,Dalian
University of Technology, Dalian, 116024, China}
\affil[b]{\it School of Software,Dalian
University of Technology, Dalian, 116620, China}
\affil[c]{\it Department of Mathematics \& Interdisciplinary Program in
Computational Sciences \& Technology, Seoul National University,
Seoul 151-747, Korea.}

\maketitle
\begin{abstract}
A new nonconforming rectangle element with cubic convergence for the energy norm is
introduced. The degrees of freedom (DOFs) are defined by the twelve values at the three
Gauss points on each of the four edges. Due to the existence of one linear relation among
the above DOFs, it turns out the DOFs are eleven. The nonconforming element consists of
$P_2\oplus \Span\{x^3y-xy^3\}$. We count the corresponding dimension for Dirichlet and
Neumann boundary value problems of second-order elliptic problems. We also present the
optimal error estimates in both broken energy and $L_2(\O)$ norms. Finally, numerical
examples match our theoretical results very well.
    \\[6pt]
\end{abstract}

\section{Introduction}
It has been well known that the standard lowest order conforming elements can
produce numerical locking and checker-board solutions in the approximation of
solid and fluid mechanics problems: see for instance \cite{bs92b, bs92a, brezzi-fortin,
  ciar, gira-ravi} and the references therein. An efficient approach to
deal with this case is to employ the nonconforming element method,
which has made a great impact on the development of finite element methods \cite{cdssy, cdy99, cr73,
fortin-soulie-quad-nc-2d, fortin3d, farhoul-fortin,
rannacher-turek,arnold-winter-03, brenner-sung, fal91, lls-nc-elast,
luskin-98, luskin-96, ming-shi-01,
zhang-quad-97,2007-Hu-p88-102,2005-Yi-p115-133,2009-Awanou-p49-60}.

To approximate PDEs using a nonconforming element of order $k$, one needs to impose the
continuity of the moments up to order $k-1$ of the functions across all the
interfaces of neighboring elements. This condition is known as the patch test
\cite{1977-Irons-p557}. In two dimensions, the patch test is equivalent to the continuity
at the $k$ Gauss points located on each interface. This implies that a $P_k$-nonconforming
element, if exists, must be continuous at the $k$ Gauss points on each edge. These points
(completed with internal points for $k\geq 3$) can be used to define local Lagrange
degrees of freedom (DOFs) on the simplex if $k$ is odd, but this construction is not
possible if $k$ is even since there exists a lower-degree polynomial vanishing at all the
Gauss points \cite{fortin-soulie-quad-nc-2d}. Thus suitable bubble functions are often employed
to enrich the finite element space. Until now, the triangular nonconforming elements are
well studied in the literature (see, \cite{cr73,fortin-soulie-quad-nc-2d}), but the analysis
of their quadrilateral counterparts is less complete.

Even though the triangular or tetrahedral meshes are popular to use, in some cases where
the geometry of the problem has a quadrilateral nature, one wishes to use quadrilateral
or hexahedral meshes with proper elements. For even $k$, the same trouble exists, that
is, there also exists a lower-degree polynomial vanishing at all the Gauss points. Again,
some bubble functions are added to the finite element space \cite{hleesheen-p2rec}.
Compared to the triangular case, another trouble for quadrilateral finite element is that
the DOFs and corresponding polynomial space do not match. Usually the number of DOFs is
bigger than the dimension of $P_k$. For example, for $k=1$ and $k=2$, the numbers of DOFs
are $4$ and $8$, respectively, while the corresponding dimensions of $P_k$ are $3$ and
$6$, respectively. Therefore, some additional relations must be imposed or some special
functions are added to the finite element space such that unisolvency can be
satisfied, see \cite{2004-Park-p624, hleesheen-p2rec, kim}.

The purpose of this paper is to develop a $P_3$-nonconforming element on rectangular
meshes. We define the 12 Gauss points (3 Gauss points on every edge) as the DOFs. To
obtain an optimal order error estimate, the finite element space must be carefully
chosen such that any function in this space is a polynomial of degree no greater
than 3 on every edge. Meanwhile, we also notice that the values on the 12 Gauss points satisfy a linear
relation if the degree of a polynomial on every edge is no more 3, which is a little
different from the triangular mesh case. Thus we define our finite element space as
$P_3\oplus\Span\{x^3y-xy^3\}$. Therefore, the number of DOFs is locally 11. We prove
unisolvency and define three types of local and global bases, one of which is defined
associated with vertices and the other two of which are defined associated with edges.
Then we derive optimal error estimates for second-order elliptic problems in broken
energy- and $L^2$-norms. Finally, numerical examples are provided, which match
our theoretical result very well.

This paper is organized as follows. In Section 2 the $P_3$ nonconforming element is defined
on rectangular meshes. The dimensions and basis functions for Dirichlet and Neumann problems
are given. In section 3 interpolation operators are defined and optimal order error estimates
are shown. In Section 4, numerical results for the elliptic problems are presented.

\section{The $P_3$ nonconforming element on rectangular mesh}

\subsection{The $P_3$-Nonconforming Quadrilateral Elements}
Take reference element as square $\hR=[-1,1]^2$. Denote by
$V_j,j=1,\cdots,4 ,$ the vertices and by
$g_{3j-2},g_{3j-1}$ and $g_{3j}$ the Gauss points on the edges
$\overline{V_jV_{j+1}},j=1,\cdots,4,$ with the identification $V_1=V_5$ (see Fig.
\ref{fig:rec}):
\begin{eqnarray*}
g_1=\left(-\gss,-1\right),&g_2=(0,-1), &g_3=\left(\gss,-1\right),\\
g_4=\left(1,-\gss\right), &g_5=(1,0),&g_6=\left(1,\gss\right),\\
g_7=\left(\gss,1\right),&g_8=(0,1), &g_9=\left(-\gss,1\right),\\
g_{10}=\left(-1,\gss\right),&g_{11}=(-1,0), &g_{12}=\left(-1,\gss\right).
\end{eqnarray*}

\begin{figure}[h]
    \begin{center}
        \includegraphics{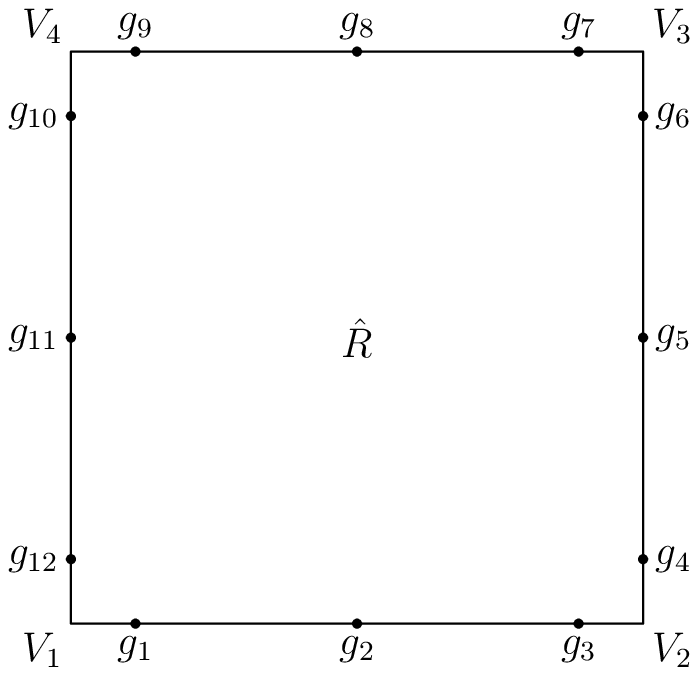}
    \end{center}
    \caption{The reference rectangle $\hR$}
    \label{fig:rec}
\end{figure}

Denoting by $P_k(\hR)$ the space of polynomials of degree $\leq k$ on
$\hR$, set
\begin{eqnarray*}
    \hbP=P_3(\hR)\oplus \Span\{x^3y-xy^3\}, \quad\text{with }\dim(\hbP) = 11.
\end{eqnarray*}
The space $\hbP$ will be our nonconforming finite element space on $\hR$ with
appropriate degrees of freedom that will be defined soon.
Before proceeding, we notice the following simple result.
\begin{lemma}
The following relationship holds:
    \begin{eqnarray}
        3p(-1)+3p(1)-5p(\gss)-5p(-\gss)+4p(0)=0 \quad\forall p \in
        P_3(\mathbb R).
        \label{eq:relation1}
    \end{eqnarray}
    \label{lem:relation1}
\end{lemma}
\begin{proof} Let $p\in P_3(\mathbb R).$ Then
    any fourth order difference quotient of $p(x)$ vanishes. The fourth order difference quotient of
    $p(x)$ at points $x_1=-1,x_2=-\gss,x_3=0,x_4=\gss$, and $x_5=1$ can be expressed by
    \begin{eqnarray*}
        \sum_{i=1}^5\frac{p(x_i)}{(x_i-x_1)\cdots
        (x_i-x_{i-1})(x_i-x_{i+1})\cdots(x_i-x_5)}.
    \end{eqnarray*}
    A simple computation derives the desired result.
\end{proof}

As an immediate consequence of Lemma \ref{lem:relation1}, we have the following proposition.
\begin{prop}\label{prop:relation}
    The following relationship holds:
    \begin{eqnarray}
        \label{eq:relation2}
        \begin{split}
        &4\big(\varphi(g_2)+\varphi(g_8)\big)-5\big(\varphi(g_1)+\varphi(g_3)+\varphi(g_7)+\varphi(g_9)\big)\\
        &\qquad=4\big(\varphi(g_5)+\varphi(g_{11})\big)-5\big(\varphi(g_4)+\varphi(g_6)+\varphi(g_{10})+\varphi(g_{12})\big)\quad
        \forall \varphi\in\hbP.
        \end{split}
    \end{eqnarray}
\end{prop}
\begin{proof}
    Notice that $\varphi$ is a polynomial of degree no greater than 3
    on any edge of $\hR$. The result follows immediately from Lemma \ref{lem:relation1}.
\end{proof}

\begin{lemma}
    Suppose that $\varphi\in\hbP$ vanishes at the twelve Gauss points
    $g_j,j=1,\cdots,12$. Then $\varphi=0$ in $\hR$.
    \label{lem:unisolvency}
\end{lemma}
\begin{proof}
Let $p\in P_3(\mathbb R)$ vanish at vanishes at the twelve Gauss points
$g_j,j=1,\cdots,12$.
    It is easy to check that $\varphi_1(x,y)=x^2+y^2-8/5$ vanishes at the
    eight Gauss
    points $g_{3j-2},g_{3j},j=1,2,3,4$. By a simple polynomial division,
    $\varphi(x,y)$ can be written as
    \begin{eqnarray*}
        \varphi(x,y)=\varphi_1(x,y)(a_0xy+a_1x+a_2y+a_3)+xr_1(y)+r_2(y)
    \end{eqnarray*}
    where $r_1(y)$ and $r_2(y)$ take the following form:
    \begin{eqnarray*}
        r_1(y)=\cor{-}2a_0y^3+b_1y^2+b_2y+b_3 \quad\text{and }\,
        r_2(y)=c_0y^3+c_1y^2+c_2y+c_3.
    \end{eqnarray*}
 All the coefficients $a_j,b_j,c_j$ are yet to be determined. It
    follows from $\varphi(\pm\gss,1)=\varphi_1(\pm\gss,1)=0$ that
    \begin{eqnarray*}
        \gss r_1(1)+r_2(1)=0,\quad \text{and}\quad
        -\gss r_1(1)+r_2(1)=0,
    \end{eqnarray*}
which implies that $r_1(1)=r_2(1)=0$. Similarly, if follows from $\varphi(\pm\gss,-1)=\varphi_1(\pm\gss,-1)=0$ that
    $r_1(-1)=r_2(-1)=0.$
Next, it follows from $\varphi(\pm 1,\gss)=\varphi_1(\pm 1,\gss)=0$ that
$r_1(\pm\gss)=r_2(\pm\gss)=0.$
%
%
%    which implies that $r_1$ and $r_2$ vanish at points $\pm 1$ and
%    $\pm \gss$. But
Since
$r_1$ and $r_2$ are polynomials of degree no
    more 3, which leads to $r_1(y)=r_2(y)=0$. Hence $\varphi(x,y)$ must have the
    following form
    \begin{eqnarray*}
        \varphi(x,y)=\varphi_1(x,y)(a_1x+a_2y+a_3).
    \end{eqnarray*}
Then the four additional conditions $\varphi(g_{3j-1})=0,j=1,2,3,4,$ lead to
    $\varphi(x,y)=0$, which completes the proof.
\end{proof}

Due to proposition \ref{prop:relation} and Lemma \ref{lem:unisolvency}, we
have the unisolvency result.
\begin{prop}\label{prop:unisolvency}
    A function $\varphi\in\hbP$ is uniquely determined by
    $\varphi(g_j),j=1,2,\cdots,12,$ which satisfy the relation
    \eqref{eq:relation2}.
\end{prop}

Denote by $M_j,j=1,2,3,4,$ the
four midpoints of four edges. Obviously, $M_j$ is one of three Gauss points on
$j$th edge and hence the other two Gauss points on $j$th edge can also be denoted
by $M_j^+$ and $M_j^-$. For example, in Fig.
\ref{fig:rec}, $M_1^+=g_1$ and $M_1^-=g_3$.  For $1\leq j\leq 4$, define $\varphi_j^V$,$\varphi_j^{E_+}$ and
$\varphi_j^{E_-}\in \hbP$ (see Fig. \ref{fig:basis}) by
\begin{eqnarray*}
    \varphi_j^V(g_k)=\begin{cases}
        1, & k=3j-3,3j-2,\\
        0, & otherwise,
    \end{cases}
\end{eqnarray*}
(with the identification $g_0=g_{12}$), and
\begin{eqnarray*}
    \varphi_j^{E_+}(g_k)=\begin{cases}
        4, & g_k=M_j^+,\\
        5, & g_k=M_j,\\
        0, & otherwise,
    \end{cases}
\end{eqnarray*}
and
\begin{eqnarray*}
    \varphi_j^{E_-}(g_k)=\begin{cases}
        5, & g_k=M_j,\\
        4, & g_k=M_j^-,\\
        0, & otherwise.
    \end{cases}
\end{eqnarray*}

\begin{figure}[htpb]
    \centering
    \subfigure[]{
        \begin{minipage}[t]{.5\textwidth}
       \centering
       \includegraphics[width=0.6\textwidth]{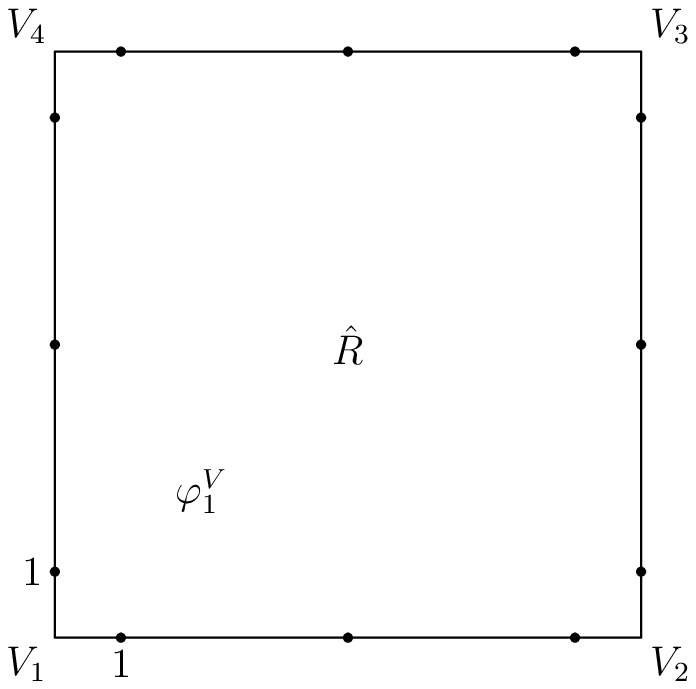}
       \end{minipage}}\\%\\
        \subfigure[]{
        \begin{minipage}[t]{.5\textwidth}
       \centering
       \includegraphics[width=0.6\textwidth]{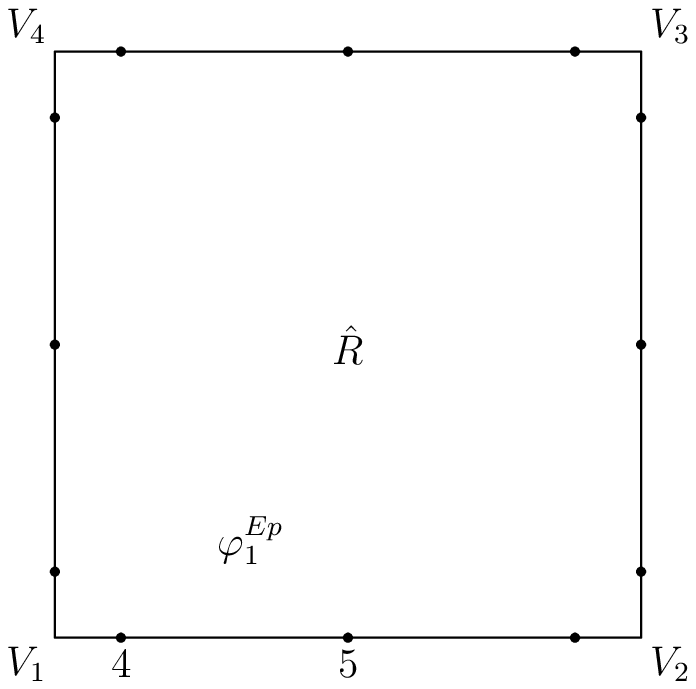}
       \end{minipage}}%
        \subfigure[]{
        \begin{minipage}[t]{.5\textwidth}
       \centering
       \includegraphics[width=0.6\textwidth]{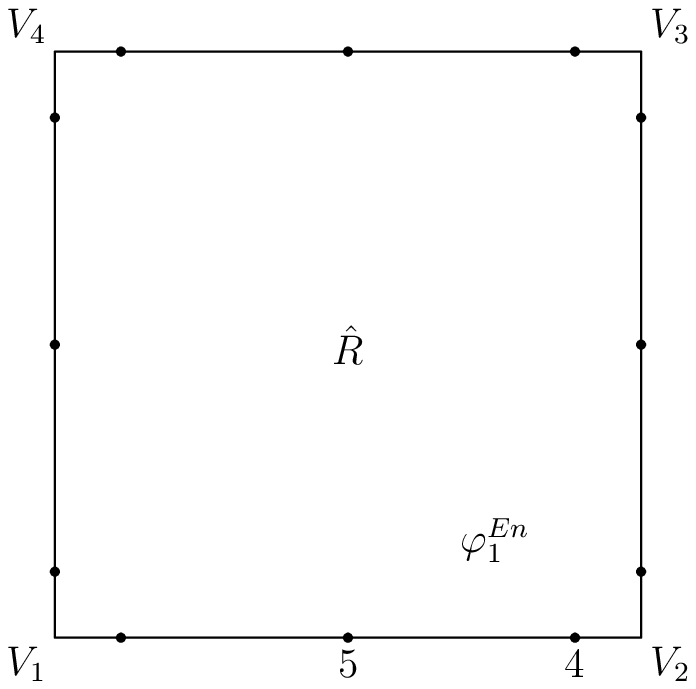}
       \end{minipage}}%
       \caption{Three types of basis function. (a) The vertex-based basis function associated with vertex $V_1$.
       (b) The fist type of the edge-based function associated with edge
       $E_1$. (c) The second type of the edge-based function associated with edge
       $E_1$.}
       \label{fig:basis}
\end{figure}

We then have the following result.
\begin{lemma}
    $\Span\{\varphi_j^V,\varphi_j^{E_+},\varphi_j^{E_-},j=1,2,3,4\}=\hbP$.
    Indeed, any eleven of
    $\varphi_j^V,\varphi_j^{E_+},\varphi_j^{E_-},j=1,2,3,4,$ span
    $\hbP$.
\end{lemma}
The proof of this lemma is similar to that of Lemma 2.3 in
\cite{hleesheen-p2rec}, which will be omitted.

We are now in a position to state the definition of $P_3$-nonconforming element
on a rectangle as follows.

\begin{defn}
    The $P_3$-nonconforming element on rectangle is defined by
    $(\hR,\hbP_{\hR},\hat{\sum}_{\hR})$, where
    \begin{itemize}
    \item
        $\hR$ is a rectangle,
    \item
        $\hbP_{\hR}=\hbP$ is the finite element space, and
    \item
        $\hat{\sum}_{\hR}=\{\varphi(g_j),j=1,2,\cdots,12,\ \text{such
        that Eq. \eqref{eq:relation2} holds for all}\ \varphi\in
        \hbP_{\hR}$\} is the degrees of freedom.
    \end{itemize}
\end{defn}
The following patch test lemma is immediate since any $p\in\hbP$ is of degree
$\le 3$ on an edge.
\begin{lemma}\label{lem:orth}
    Let $E_j$ denote the face containing the midpoint $M_j$ and the other
    two Gauss points $M_j^+$ and $M_j^-.$
    If
    $p\in\hbP,$  $p(M_j)=p(M_j^+)=p(M_j^-)=0$, then
    \begin{eqnarray}\label{eq:orth_property1}
    \int_{E_j}p(x,y)q(x,y)\,\d \sigma=0
    \end{eqnarray}
for all $q\in Q_2(E_j).$
\end{lemma}

\begin{remark}
    For actual computation, the local finite element can be alternatively
    given by
    \{$\sum_{j=1}^4\big(c_j^V\varphi_j^V+c_j^{E_+}\varphi_j^{E_+}+c_j^{E_-}\varphi_j^{E_-}$\big)\}.
\end{remark}
\begin{remark}
    In theory, if $x^3y-xy^3$ is replaced by $x^3y$, $xy^3$ or $x^3y+xy^3$, the
    unisolvency like Proposition \ref{prop:unisolvency} holds. But the
    former two choices are lack in symmetry. As for the third choice $x^3y+xy^3$,
    it turns out to be numerically during the computation of the bases since the
    corresponding coefficient matrices are nearly
    ill-conditioned (the determinants are near to zero).
\end{remark}

Let us proceed to define our $P_3$-nonconforming element space. Assume that
$\O\in \R^2$ is a parallelogram domain with boundary
$\Gamma$. Let $(\Tau_h)_{h>0}$ be a regular family of triangulation of $\O$ into
parallelograms $R_j,j=1,2,\cdots,N_R$, where $h=\max_{R\in \Tau_h}h_R$ with
$h_R=\text{diam}(R)$. For each $R\in\Tau_h$, let $F_R: \hR\rightarrow
\mathbb R^2$ be an invertible affine mapping such that
\[
R = F_R(\hR),
\]
and denote $\hphi_{R} = \hphi\circ F_R^{-1}: R \rightarrow \mathbb R$ for
all $\hphi \in \hbP,$ whose collection will be designated by
\[
\bP_R = \Span\{\phi_{R},\,\hphi \in \hbP \}.
\]

For a given triangulation $\Tau_h$ of $\O$, let $N_V,N_E,N_R$, and $N_G$
denote the numbers of vertices, edges, rectangles, and Gauss points,
respectively. Then set
\begin{eqnarray*}
    && \V_h=\{V_1,V_2,\cdots,V_{N_V}\}: \ \text{the set of all vertices
    of}\ R\in\T_h,\\
    && \E_h=\{E_1,E_2,\cdots,E_{N_E}\}: \ \text{the set of all edges of}\ R\in\T_h,\\
    && \G_h=\{g_1,g_2,\cdots,g_{N_G}\}:\ \text{the set of all Gauss points
    on}\ \E_h\in\T_h\\
    && \M_h=\{M_i\in E_i,i=1,2,\cdots,N_E\}:\ \text{the set of all Midpoints on}\ \E_h\in\T_h.
\end{eqnarray*}
In particular, let $N^i_V, N^i_E$, and $N^i_G$ denote the numbers of interior vertices, edges, and Gauss
points of $R\in \T_h$, respectively.

For a function $f$ defined in $\O$, denote by $f_j$ its restriction to $R_j$,
and $E_{jk}$ the interface between $R_j$ and $R_k$. Similarly,
$g_{jk},k=1,2,3,$ will mean the Gauss points on $\Gamma_j=\partial R_j\cap
\partial \O$ and $g_{jkl}, l = 1, 2,3,$ will be the Gauss points on $E_{jk}$.
We are now in a position to define the following nonconforming
finite element spaces.
\begin{eqnarray*}
    &&\NC=\{\varphi:\O\rightarrow \R\big|\varphi|_R\in \bP_R,\forall
    R\in\T_h, \varphi\ \text{is continuous at the Gauss points } g\in\G_h\},\\
    &&\NC_0=\{\varphi\in\NC\big|\varphi(g)=0, \text{for all Gauss points }\ g\in
    \G_h\cap \Gamma\}.
\end{eqnarray*}

For each vertex $V_j\in\V_h$, denote by $\E_h(j)$ and $\G_h(j)$ the set of all
edges $E\in \E_h$ with one of the endpoints being $V_j$ and the set of Gauss points
nearest to $V_j$ among the three Gauss points on $E$ for all
$E\in\E_h(j).$ For $M_j\in E_j$, if $g_i$ and $g_k$ are two
other Gauss points and $i<k$, we also denote these two Gauss points by $M_j^+$
and $M_j^-$, respectively.
We then define the three types of functions in $\NC$, which serve as global
bases for the nonconforming finite element spaces.

\begin{defn}\label{def:basis}
    The first type of functions are associated with vertices. Define
    $\varphi_j^V\in\NC,j=1,2,\cdots,N_V$, by
    \begin{eqnarray*}
        \varphi_j^V=\begin{cases}
            1, & g_k\in \G_h(j)\\
            0, & \text{otherwise}
        \end{cases}
    \end{eqnarray*}
    Next define the second type of functions associated with edges
    $E_j\in\E:$ define $\varphi_j^{E_+}\in\NC,j=1,2,\cdots,N_E$, by
    \begin{eqnarray*}
        \varphi_j^{E_+}=\begin{cases}
            5, & g_k=M_j\\
            4, & g_k=M_j^+\\
            0, & \text{otherwise}
        \end{cases}
    \end{eqnarray*}
    The last type of functions are also associated with edges
    $E_j\in\E:$ define $\varphi_j^{E_-}\in\NC,j=1,2,\cdots,N_E$, by
    \begin{eqnarray*}
        \varphi_j^{E_-}=\begin{cases}
            5, & g_k=M_j\\
            4, & g_k=M_j^-\\
            0, & \text{otherwise}
        \end{cases}
    \end{eqnarray*}
\end{defn}
Similarly, define the three types of functions which will serve as global
basis functions for $\NC_0$ with those for $\NC$ excluding $\varphi_j^V$'s
which are associated with boundary vertices and
$\varphi_j^{E_+},\varphi_j^{E_-}$'s which
are associated with boundary edges.

Now let us present the dimensions for the nonconforming finite element spaces.

\begin{thm}
    $\dim(\NC)=N_V+2N_E-1$. Let $\varphi_j^V,j=1,2,\cdots,N_V$ and
    $\varphi_j^{E_+},\varphi_j^{E_-},j=1,2,\cdots,N_E$ be the functions
    defined in Definition \ref{def:basis}. By omitting any one of
    these functions, each of the following three sets forms
 global basis functions for $\NC$:
    \begin{align}
        &\B_1=\{\varphi_1^V,\varphi_2^V,\cdots,\varphi_{N_V-1}^V,\varphi_j^{E_+},\varphi_j^{E_-},j=1,2,\cdots,N_E\},\\
        &\B_2=\{\varphi_j^V,j=1,2,\cdots,N_V,\varphi_1^{E_+},\varphi_2^{E_+},\cdots,\varphi_{N_E-1}^{E_+},\varphi_i^{E_-},i=1,2,\cdots,N_E\},\\
\intertext{and}
        &\B_3=\{\varphi_j^V,j=1,2,\cdots,N_V,\varphi_i^{E_+},i=1,2,\cdots
        N_E,\varphi_1^{E_-},\varphi_2^{E_-},\cdots,\varphi_{N_E-1}^{E_-}\}.
    \end{align}
    \label{thm:dim}
\end{thm}
\begin{thm}
    $\dim(\NC_0)=N_V^i+2N_E^i$.
    $\B=\{\varphi_i^V,i=1,2,\cdots,N_V^i,\varphi_j^{E_+},\varphi_j^{E_-},j=1,2,\cdots,N_E^i\}$
    forms a set of global basis functions for $\NC_0$.
\end{thm}
The proofs of the above theorems are quite similar to those in the
literature \cite{hleesheen-p2rec}, and thus omitted. Here we remark that our
finite element space is a little different from those in the literature
\cite{hleesheen-p2rec, kim}. Those finite element spaces
are nothing but conforming element spaces enriched by some suitable bubble
function spaces. Thus the idea of Fortin and Soulie
\cite{fortin-soulie-quad-nc-2d}
is not applicable here.

\section{The interpolation operator and convergence analysis}
In this section we will define an interpolation operator and analyze convergence in the case of
Dirichlet problem. The case of Neumann problem is quite similar and the results will be briefly
stated with their details being omitted.

Denote by $(\cdot,\cdot)$ the $L^2(\O)$ inner product and $(f,v)$ will be understood as the duality pairing between $H^{-1}(\O)$ and $H^1_0(\O)$, which
is an extension of the duality paring between $L^2(\O)$. By $\|\cdot\|_k$ and
$|\cdot|_k$ we adopt the standard notations for the norm and seminorm for the Sobolev spaces $H^k(\O)$.
Consider then the following Dirichlet problem:
\begin{subequations}\label{eq:elliptic}
\begin{eqnarray}
    -\nabla\cdot\left(\balpha\nabla u\right)+\beta u&=&f, \quad \O, \\
    u&=&0,\quad \Gamma,
\end{eqnarray}
\end{subequations}
with $\balpha=(\alpha_{jk}),\alpha_{jk},\beta\in L^{\infty}(\O), j,k = 1,2, 0 <\alpha_*|\bxi|^2\leq\bxi^t\balpha(x)\bxi\leq\alpha^*
|\bxi|^2 <\infty,\bxi\in\R^2, \beta(x)\geq 0, x\in\O$, and $f\in H^1(\O)$.
We will assume that the coefficients
are sufficiently smooth and that the elliptic problem
\eqref{eq:elliptic} has an $H^4(\O)$-regular solution.
 The weak problem is then given as usual:
find $u\in H_0^1(\O)$ such that
\begin{equation}
a(u, v) = (f,v), \quad v\in H^1_0(\O),
    \label{eq:weak}
\end{equation}
where $a: H^1_0(\O)\times H^1_0(\O) \rightarrow \R$  is the bilinear
form defined by $a(u,v) = (\balpha \nabla u,\nabla v) + (\beta u,
v)$ for all $u,v\in H^1_0(\O)$. Our nonconforming method for Problem
\eqref{eq:elliptic} states as follows: find $u_h\in\NC_0$ such that
\begin{equation}
    a_h(u_h, v_h) = (f, v_h),\quad v_h\in\NC_0,
    \label{eq:solution}
\end{equation}
where
$$
a_h(u, v) =\sum_{R\in \Tau_h}a_R(u,v),
$$
with $a_R$ being the restriction of $a$ to $R$.

For a given rectangle $R\in \Tau_h$, define the local interpolation operator
$\Pi_R : W^{1,p}(R)\cap H^1_0(\O)\longrightarrow \bP_R,~ p>1,$ by
$$
\Pi_R\phi(g_i)=\phi(g_i),
$$
for all Gauss points on the edges of $R$.
The global interpolation operator $\Pi_h$: $W^{1,p}(\O)\cap H^1_0(\O)\rightarrow\NC_0$
is then defined through the local interpolation operator $\Pi_R$ by $\Pi_h|_R=\Pi_R$ for all $R\in\Tau_h$. Since $\Pi_h$
preserves $P_3$ for all $R\in\Tau_h$, it follows from the Bramble-Hilbert Lemma that
\begin{eqnarray}\label{eq:hilbert}
    \begin{split}
\sum_{R\in\Tau_h}\|\phi-\Pi_h\phi\|_{L^2(R)}+h\sum_{R\in\Tau_h}\|\phi-\Pi_h\phi\|_{H^1(R)}\le
Ch^k|\phi|_{H^k(\O)},
\\
\phi\in W^{k,p}(\O)\cap H_0^1(\O),\, 2\le k\le 4.
\end{split}
\end{eqnarray}
Denote $E_{jk}=(\pa R_j\cap \pa R_k)^{\circ}$ for all $R_j,R_k\in\Tau_h$ and by
$\Gamma_j$ the boundary face $(\pa R_j\cap \pa\O)^{\circ}$ of $\Tau_h$.
 Then define
 $$\Lambda^h=\{ \lambda|\lambda_{jk}=\lambda|_{E_{jk}}\in
P_2(E_{jk});\lambda_{jk}+\lambda_{kj}=0;~\lambda_j=\lambda|_{\Gamma_j}\in
P_2(\Gamma_j)\},$$ where $P_2(E)$ denotes the set of quadratic polynomials on the face
$E$. Also define the projection $P_h:H^{3/2}(\O)\rightarrow \Lambda^h$
 such that
 \begin{eqnarray*}
\bigg<
\bnu^T_j\balpha\nabla v_j-P_hv_j,z
\bigg>_E=0\quad \mbox{for all}~z\in P_2(E)~\mbox{for all}~E\in \E_h,
 \end{eqnarray*}
where $v_j=v|_{R_j}$ and $\bnu _j^T$ is the transpose of the unit outward normal to $R_j$. Then we
have the following standard polynomial approximation result:
\begin{eqnarray}\label{eq:projection}
\bigg\{
\sum_j\|\bnu^T_j\balpha\nabla v_j-P_h v_j\|_{L^2(\pa R_j)}^2
\bigg\}^{1/2}
\le Ch^{k-3/2}\|v\|_k,\quad \forall v\in W^{k,p}(\O),k=2,3,4,p>3.
\end{eqnarray}
Since $w_j-w_k$ has zero values at the Gauss points on $E_{jk}$ for all
$w\in\NC_0$ and the 3-point Gauss quadrature is exact on polynomials of degree
no more than 5, the following useful orthogonality holds. (See also Lemma \ref{lem:orth})
\begin{lemma}
If $u\in H^{3/2}(\O)$, then the following equality holds:
\begin{eqnarray*}
\langle P_h u_j,w_j\rangle_{E_{jk}}+\langle P_h
u_k,w_k\rangle_{E_{kj}}=\langle P_h u_j,w_j-w_k\rangle_{E_{jk}}=0
\quad \mbox{for all}~w\in\NC_0.
\end{eqnarray*}
\end{lemma}

%Furthermore, employing the following notation
%\begin{eqnarray*}
    %\tilde{\Lambda}=\big\{\Pi_{E=\partial R_j\cap\partial
    %R_k\in\E_h^i}\big(w_h|_{\partial R_j\cap E},w_h|_{\partial R_k\cap
    %E}\big)\times \Pi_{E\in \E_h^b}w_h|_E, \forall w_h\in\NC\big\}
%\end{eqnarray*}
%Denote by $\Pi_E^2$: $\tilde{\Lambda}\rightarrow \Lambda^h$ the interpolation such that
%$\Pi_E^2(w_h|_E)$ and $w_h$ coincide at the three Gauss points on $E$ for
%every edge $E$ in $\T_h$ for all $w_h\in\NC$. Then we have

Denote the broken energy norm $\|\cdot\|_h$ on $\NC+H^1(\O)$ by
\begin{eqnarray*}
    \|\varphi \|_h=a_h(\varphi,\varphi)^{1/2}\quad \text{for all}\
    \varphi\in \NC+H^1(\O).
\end{eqnarray*}
We now consider the energy-norm error estimate and first consider the
following Strang lemma \cite{strang-fix}.
\begin{lemma}
    Let $u\in H^1(\O)$ and $u_h\in \NC_0$ be the solutions of Eq.
    \eqref{eq:weak} and Eq. \eqref{eq:solution}, respectively.
Then
\begin{eqnarray}\label{eq:strang}
    \|u-u_h\|_h\leq c\Big\{\inf_{v\in \NC_0}\|u-v\|_h+\sup_{w\in\NC_0}
    \frac{|a_h(u,w)-\langle f,w\rangle|}{\|w\|_h}\Big\}.
\end{eqnarray}
    \label{lem:strang}
\end{lemma}

Assume sufficient regularity such that $u\in H^4$. Due to \eqref{eq:hilbert}, the first term in the right side
of \eqref{eq:strang} is bounded by
\begin{eqnarray}\label{eq:int_error}
    \inf_{v\in \NC_0}\|u-v\|_h\leq \|u-\Pi_h u\|_h\leq
    ch^s|u|_{H^{s+1}(\O)},1<s\leq 3
\end{eqnarray}

In order to bound the second term of the right side of \eqref{eq:strang} which denotes the consistency error,
integrate by parts elementwise so that
\begin{eqnarray*}
    a_h(u,w)-\langle f,w\rangle&=&\sum_j\langle \bnu^T_j\balpha\nabla
    u_j,w\rangle_{\partial R_j\backslash \Gamma_j}\\
    &=& \sum_j\langle \bnu^T_j\balpha\nabla
    u_j-P_h u_j,w\rangle_{\partial R_j\backslash \Gamma_j}\\
    &=& \sum_j\langle \bnu^T_j\balpha\nabla
    u_j-P_h u_j,w-m_j\rangle_{\partial R_j\backslash \Gamma_j}\\
\end{eqnarray*}
where $m_j\in Q_2(R_j)$ is a biquadratic polynomial on $R_j$. In particular, if $m_j$ is
chosen as the $Q_2$ projection of $w$ in $R_j$, due to the trace theorem,
\eqref{eq:hilbert} and \eqref{eq:projection}, we get
\begin{eqnarray}
    \begin{split}
    \bigg|\sum_j\langle \bnu^T_j\balpha\nabla
    u_j-P_h u_j,w-m_j\rangle_{\partial R_j}\bigg|
    &\leq \sum_j \|\bnu^T_j\balpha\nabla u_j-P_h u_j\|_{L_2(\partial R_j)}
    \|w-m_j\|_{L_2(\partial R_j)}\\
    &\leq \Big(\sum_j \|\bnu^T_j\balpha\nabla u_j-P_h u_j\|^2_{L_2(\partial
    R_j)}\Big)^{1/2}\Big(\sum_j \|w-m_j\|^2_{L_2(\partial
    R_j)}\Big)^{1/2}\\
    &\leq Ch^{k-3/2}\|u\|_k \Big(\sum_j
    \|w-m_j\|_{L_2(R_j)}\|w-m_j\|_{H_1(R_j)}\Big)^{1/2}\\
    & \leq Ch^{k-3/2}\|u\|_k h^{1/2}\|w\|_h\\
    &=Ch^{k-1}\|u\|_k \|w\|_h,\quad k=2,3,4.
    \end{split}
    \label{eq:consistent}
\end{eqnarray}
Thus, we have
\begin{eqnarray*}
    \sup_{w\in\NC_0}\Big\{
    \frac{|a_h(u,w)-\langle f,w\rangle|}{\|w\|_h}\Big\}\leq
    Ch^{k-1}\|u\|_{H^k(\O)}, \quad k=2,3,4.
\end{eqnarray*}
By collecting the above results, we get the following energy-norm error estimate.
\begin{thm}
    Let $u\in H^{k+1}(\O)\cap H_0^1(\O)$ and $u_h\in\NC_0$ be the solution of
    \eqref{eq:weak} and \eqref{eq:solution}, respectively. Then we have
    \begin{eqnarray*}
        \|u-u_h\|_h\leq ch^k \|u\|_{H^{k+1}(\O)},\quad k=1,2,3.
    \end{eqnarray*}
\end{thm}

By the standard Aubin-Nitsche duality argument, the $L_2(\O)$-error
estimate can be easily obtained, but the corresponding proof is omitted. We
state the result in the following theorem.
\begin{thm}
    Let $u\in H^{k+1}(\O)\cap H_0^1(\O)$ and $u_h\in\NC_0$ be the solution of
    \eqref{eq:weak} and \eqref{eq:solution}, respectively. Then we have
    \begin{eqnarray*}
        \|u-u_h\|_{0}\leq Ch^{k+1} \|u\|_{H^{k+1}(\O)},\quad k=1,2,3.
    \end{eqnarray*}
\end{thm}

Instead of the Dirichlet problem, if the following Neumann problem
\begin{subequations}
    \begin{eqnarray}
        -\nabla\cdot \left( \balpha \nabla u \right)+\beta u=f,\quad \O, &&\\
       \nu\cdot \left(\balpha \nabla u \right) +\gamma u=g,\quad \Gamma,&&
    \end{eqnarray}
    \label{eq:Neuman}
\end{subequations}
is considered, the weak problem \eqref{eq:weak} is then replaced by finding $u\in H^1(\O)$ such that
\begin{equation}
    a^n(u,v)=(f,v)+\langle g,v\rangle,\quad v\in H^1(\O),
    \label{eq:weakN}
\end{equation}
where $a^n$ is the bilinear form defined by $a^n(u,v)=(\balpha\nabla u,\nabla
v)+(\beta u,v)+\langle\gamma u,v\rangle$ for all $u,v\in H^1(\O)$, and
$\langle\cdot,\cdot\rangle$ is the paring between $H^{-1/2}(\Gamma)$ and $H^{1/2}(\Gamma)$.
Thus, the nonconforming method
for Problem \eqref{eq:Neuman} states as follows: find $u_h\in \NC$ such that
\begin{equation}
    a^n_h(u_h,v_h)=(f,v_h)+\langle g,v_h\rangle,\quad v_h\in \NC.
    \label{eq:Nsolution}
\end{equation}
Then all the arguments given above for Dirichlet case hold analogously
Hence one can have the following result.
\begin{thm}\label{th:Neu}
    Let $u\in H^{k+1}(\O)$ and $u_h\in\NC$ satisfy \eqref{eq:weakN} and \eqref{eq:Nsolution}, respectively.
    Then we have the energy-norm and $L^2$-norm error estimates:
    \begin{eqnarray*}
        &&||u-u_h||_h\leq Ch^k||u||_{H^{k+1}(\O)},\\
        &&||u-u_h||_0\leq Ch^{k+1}||u||_{H^{k+1}(\O)},\ k=1,2,3.
    \end{eqnarray*}
\end{thm}

\section{Numerical examples}
In this section we illustrate two numerical examples. First, consider the following Dirichlet
problem:
\begin{eqnarray*}
    -\Delta u=f, \quad \O, \\
    u=0,\quad \Gamma,
\end{eqnarray*}
where $\O=(0,1)^2$. The source term $f$ is calculated from the the exact
solution $$u(x,y)=\sin(2\pi x)\sin(2\pi y)(x^3-y^4+x^2y^3).$$
Table \ref{table:dirichlet} shows the numerical results, where the error reduction
ratios in $L_2(\O)$ and broken energy norm are optimal.

\begin{table}[hbt]
\begin{center}
\begin{tabular}{|r|c|c|c|c|c|}\hline\hline
h & DOFs& $\|u-u_h \|_0$ & ratio &  $\|u-u_h \|_h$ & ratio  \\ \hline
${1}\slash{2}$ & 9 &  0.148  & - &    1.759   & -  \\
${1}\slash{4}$ & 57 &  1.200E-002 & 3.62 & 0.300   & 2.55  \\
${1}\slash{8}$ & 273 &  4.690E-004     & 4.68 &  3.051E-002  &3.30  \\
${1}\slash{16}$ & 1185 & 2.292E-005   & 4.35  &   3.355E-003  &3.19  \\
${1}\slash{32}$ & 4929 & 1.279E-006 & 4.16 &  3.940E-004   &3.09  \\
${1}\slash{64}$ & 20097 & 7.590E-008 & 4.08  &  4.78E-005  &3.04  \\
${1}\slash{128}$ & 81153 &   4.629E-009    & 4.04 & 5.881E-006 &3.02 \\
\hline\hline
\end{tabular}
\end{center}
\caption{\label{table:dirichlet} The Dirichlet problem: The apparent $L^2$
  and broken energy norm errors and their reduction ratios on the
  quadrilateral meshes.}
\end{table}

Next, turn to the following Neumann problem:
\begin{eqnarray*}
    -\Delta u+u=f,\quad \O,\\
    \frac{\partial u}{\partial n}=g,\quad \Gamma,
\end{eqnarray*}
with the same domain $\O=(0,1)^2$. The source terms $f$ and $g$ are
are generated from the exact solution
\begin{eqnarray*}
    u(x,y)=\cos(2\pi x)\cos(2\pi y)(x^3-y^4+x^2y^3).
\end{eqnarray*}
Again, Table \ref{table:neumann} shows the numerical
results, where the error reduction ratios in $L_2(\O)$ and broken energy norm are optimal.

\begin{table}[hbt]
\begin{center}
\begin{tabular}{|r|c|c|c|c|c|}\hline\hline
h & DOFs& $\|u-u_h \|_0$ & ratio &  $\|u-u_h \|_h$ & ratio  \\ \hline
${1}\slash{2}$ & 32 &  3.850E-002   & - &    0.698   & -  \\
${1}\slash{4}$ & 104 &  5.217E-003 & 2.88 & 0.172   & 2.02  \\
${1}\slash{8}$ & 368 &  3.325E-004     & 3.97 &  2.348E-002  &2.88  \\
${1}\slash{16}$ & 1376 & 1.917E-005   & 4.12   &   2.907E-003  &3.01   \\
${1}\slash{32}$ & 5312 & 1.162E-006 & 4.04  &  3.616E-004   &3.01  \\
${1}\slash{64}$ & 20864 & 7.201E-008 & 4.01  &  4.513E-005  &3.00  \\
${1}\slash{128}$ & 82688 &   4.491E-009    & 4.00 & 5.639E-006 &3.00 \\
%${1}\slash{256}$ & 263168 &  0.628062E-07    & 2.99 & 0.112578E-03  &2.00
\hline\hline
\end{tabular}
\end{center}
\caption{\label{table:neumann} The Neumann problem: The apparent $L^2$-
  and broken energy norm errors and their reduction ratios on the
  quadrilateral meshes.}
\end{table}

\bibliographystyle{plain}
%\bibliography{bib,fem}

\end{document}